\documentclass[psamsfonts,12pt]{amsart}
\usepackage{latexsym}
\usepackage{amsthm}
\usepackage{amsmath}
\usepackage[dvips]{graphics}
\usepackage{graphicx}
\usepackage{amssymb}
\usepackage{color}
\usepackage{ulem}
\usepackage{epsfig}
\usepackage[all]{xy}

\numberwithin{equation}{section}

\theoremstyle{plain}
\newtheorem{theorem}[equation]{Theorem}
\newtheorem{lemma}[equation]{Lemma}
\newtheorem{corollary}[equation]{Corollary}

\newtheorem{proposition}[equation]{Proposition}
\newtheorem{remark}[equation]{Remark}
\newtheorem{algorithm}[equation]{Algorithm}
\newtheorem{example}[equation]{Example}

\theoremstyle{definition}

\newcommand{\Magma}{{\sf Magma}}
 
\def\N{{\mathbb N}}
\def\Z{{\mathbb Z}}

\begin{document}
\title{The aliquot constant}
\author{Wieb Bosma}
\author{Ben Kane}
\address{Department of Mathematics, Radboud Universiteit, Nijmegen NL}
\email{\{W.Bosma, B.Kane\}@math.ru.nl}
\date{\today}
\begin{abstract}
The average value of $\log s(n)/n$ taken over the first $N$ even integers
is shown to converge to a constant $\lambda$ when $N$ tends to infinity; 
moreover, the value of this constant
is approximated and proven to be less than $0$. Here $s(n)$
sums the divisors of $n$ less than $n$. Thus
the geometric mean of $s(n)/n$, the growth factor of the function $s$, in the long run
tends to be less than $1$. This could be interpreted as probabilistic
evidence 
that aliquot sequences tend to remain bounded.
\end{abstract}
\maketitle

\section{Introduction}
\noindent
This paper is concerned with the average growth of aliquot sequences.
An aliquot sequence is a sequence $a_0, a_1, a_2, \ldots$ of positive
integers obtained by iteration of the {\it sum-of-aliquot-divisors} function $s$,
which is defined for $n>1$ by
$$s(n)=\sum_{d|n\atop d<n} d.$$
The aliquot sequence with starting value $a_0$ is
then equal to 
$$a_0,\quad a_1=s(a_0),\quad a_2=s(a_1)=s^2(a_0), \quad\ldots;$$
we will say that the sequence terminates (at 1) if $a_k=s^{k}(a_0)=1$
for some $k\geq0$. The sequence cycles (or is said to end in a cycle)
if $s^k(a_0)=s^l(a_0)$ for some $k, l$ with $0\leq l<k$, where $s^0(n)=n$
by definition.

Note that $s$ is related to the ordinary
{\it sum-of-divisors} function $\sigma$, with
$\sigma(n)=\sum_{d|n} d$, by $s(n)=\sigma(n)-n$ for integers $n>1$. 

The main open question in this area can be phrased as: does every aliquot cycle
remain bounded? That is,
does every aliquot sequence terminate (at $1$) or cycle, or do sequences exist that
grow unbounded? The conjecture that all sequences remain bounded is often
referred to as the Catalan-Dickson conjecture.

The origin of this paper lies in computational work done 
to test integer factorization routines for the computer algebra system \Magma
\cite{Magma}.
The most efficient known method to compute $s(n)$ uses the multiplicativity
of $\sigma$ and requires the factorization of $n$. Iterating $s$ provides
long sequences of more or less random numbers of similar size, and this property
is useful in testing factorization methods. It was noticed that for even starting
values the aliquot sequences tend to increase or decrease in size fairly slowly,
by an amount that seemed constant over different starting values, whereas 
sequences with odd starting values usually terminate quickly.

Around 1996 Andrew Granville \cite{Granville} furnished a proof for this
phenomenon; see Theorem (\ref{thm:all}) and Theorem (\ref{thm:even}) below.
Further computation seemed to suggest that the constant $\lambda$ involved
would be smaller (but only just!) than $0$, but no proof of this was obtained.
Recently, we were able to obtain estimates that are good enough to prove
this property.

Altogether this led to the main result of this paper.
\begin{theorem}
The geometric mean $\mu$ of $\frac{s(2n)}{2n}$ over all positive integers $n$ exists, and equals
$$\mu=e^{-0.03\cdots}.$$
In particular, the aliquot growth factor $\mu=0.969\cdots <1$.
\end{theorem}
\noindent
Roughly summarizing, this means: on average, even aliquot sequences
tend to decrease in size! In some sense this may be taken as 
probabilistic evidence in favour of the Catalan-Dickson conjecture.

This paper is built up as follows. After some preliminaries, we state and prove
the convergence of the geometric mean for even values; also, an expression for
the resulting constant $\lambda=\log\mu$ as a difference of $\alpha$ (closely
related to the growth of $\sigma(n)/n$) and $\beta$ is
derived. In the next section an easy upper bound for $\alpha$ (which involves
a sum over all prime numbers) is given. The final section is devoted to a lower 
bound for $\beta$; this is trickier, as it involves an infinite sum of terms
that themselves are infinite products over all primes.
\section{Elementary observations}
\noindent
Although $\sigma(n)> n$ for $n>1$, all three possibilities
$s(n)<n$, $s(n)=n$ and $s(n)>n$ for $s$ do occur: $s(p)=1$ for prime numbers,
and in general $s(p^k)=1+p+\cdots+p^{k-1}<p^k$ for prime powers;
$s(P)=P$ for perfect numbers $P=2^{p-1}(2^p-1)$ (with $2^p-1$ prime),
and $s(n)>n$ for $n=P\cdot q$, where $P$ is perfect and $q$ any odd
prime other than $2^p-1$, since $\sigma(Pq)=2P(q+1)$.

Besides terminating at $1$ after hitting a prime, or ending in a perfect number,
it is also possible that an aliquot sequence ends in a cycle of length $2$ or more:
amicable numbers are pairs $m, n$ for which $\sigma(m)=m+n=\sigma(n)$, hence
$s(m)=n$ and $s(n)=m$, and a 2-cycle is formed. Sociable numbers form
cycles of larger length (and are known only for length $4$, $5$, $6$, $8$, $9$
and $28$; see \cite{Moews}). For more on these cycles, including historical remarks, see also \cite{Kobayashi}.

Note that
$$\sigma(n)=\prod_{p^k||n}{(1+p+\cdots+p^k)}$$
is multiplicative, while $s(n)=\sigma(n)-n$ is not.
A useful observation is that
\begin{equation}\label{prop:oneandahalf}
\frac{\sigma(2n)}{2n}\geq\frac{3}{2},
\end{equation}
by multiplicativity of $\sigma$ and since $\frac{\sigma(2^k)}{2^k}\geq \frac{3}{2}$
for $k\geq 1$.

Also note that $1+p+\cdots+p^k$ (for prime $p$) is only odd when
$p$ is odd and $k$ is even. Hence $\sigma(n)-n$ for odd $n$ will only be
even if $n$ is a square, and for even $n$ it will only be odd if $n$ is
a square or twice a square.
Hence: unless an accidental square (or twice a square) occurs, parity is
preserved in aliquot sequences! In this sense $s$ does not behave randomly at
all. 

In fact, divisibility by (even) perfect numbers also tends to persist, and
Guy and Selfridge \cite{GuySelfridge}  studied other types of persistence as well, but we
will ignore all but the parity aspect and only consider even and odd aliquot
sequences separately.

We are interested in the growth of the sequence
$n, s(n), s^2(n), \ldots$, in other words, in the question of
whether $s(n)/n$ tends to be smaller or greater than $1$.

Apparently, Wunderlich (in \cite{Wunderlich}) was the first to state
the following result explicitly; the first statement
(formulated for $\sigma$ rather than $s$) appears already in \cite{Dirichlet}.
\begin{theorem}\label{thm:wun}
{\bf Theorem} {\sl 
\begin{eqnarray*}
\lim_{N\rightarrow\infty}{1\over N}\sum_{n=1}^{N}{s(n)\over n}&=&{\pi^2\over 6}-1=0.6449\ldots\\
\lim_{N\rightarrow\infty}{1\over N}\sum_{n=1}^{N}{s(2n)\over 2n}&=&{5\pi^2\over 24}-1=1.0562\ldots\\
\lim_{N\rightarrow\infty}{1\over N}\sum_{n=1}^{N}{s(2n-1)\over 2n-1}&=&{3\pi^2\over 24}-1=0.2337\ldots\\
\end{eqnarray*}}
\end{theorem}
\noindent
As
\begin{equation*}
\sum_{n=1}^N{s(2n-1)\over 2n-1}= \sum_{n=1\atop n\ \mbox{\tiny odd}}^{2N}\left({\sigma(n)\over n}-1\right)
\end{equation*}
we find by the same argument usually given for the computation of $\zeta(2)$, that
$$
\lim_{N\rightarrow\infty}\sum_{n=1}^N{1\over N}{s(2n-1)\over 2n-1}=
\sum_{x=1\atop x\ \mbox{\tiny odd}}^{2N}
{1\over x^2}-1={3\over 4}\zeta(2)-1.
$$
from which the whole theorem follows.

Based on the second statement in Theorem (\ref{thm:wun}), Guy and Selfridge 
seem to have drawn the conclusion that even aliquot sequences will tend to grow
unbounded (see \cite{GuySelfridge} page 103). Just like a sequence in which the terms
are alternately multiplied by $5$ and by $\frac{1}{5}$ will remain bounded
although the average growth factor tends to $2.6$, we cannot draw the conclusion
that even aliquot sequences tend to grow unbounded from the fact that
the average of $s(2n)/2n$ exceeds 1. What really matters is 
not the arithmetic mean, but rather
the geometric mean:
$$
\root N \of {\prod_{n=1}^N{s(n)\over n}}=\exp\left({{1\over N}{\sum_{n=1}^N\log(s(n)/n)}}\right).
$$
\begin{remark}\rm
To draw conclusions about the Catalan-Dickson conjecture, one needs more
than just the arithmetic or geometric mean of $\sigma(n)/n$. Davenport \cite{Davenport}
showed that there exists a continuous function of $t$ giving 
the natural density of $t$-abundant numbers satisfying $\sigma(n)/n\geq t$.
See \cite{Weingartner}, and \cite{Petermann}, for recent progress on this
function. This needs to
be combined with the persistance of drivers as in  \cite{GuySelfridge}.

A good approximation is known for the value of the function at
$t=2$, implying that
$s(n)/n$ exceeds $1$ for a little less than a quarter of all $n$, see
\cite{Deleglise}.

Finally, note that these questions relate to deep problems, as it is
known \cite{Robin} that the Riemann hypothesis is equivalent to the
statement that $\sigma(n)/n$ is bounded by $e^\gamma \log\log n$ for
all $n\geq 5041$; see also \cite{Briggs}, \cite{Lagarias}, \cite{Wojtowicz} on this
connection.
\end{remark}
\section{The aliquot constant}
\noindent
We first show the following result on the geometric mean for the
ordinary sum of divisors function.
\begin{proposition}\label{one}
\begin{equation*}
\frac{1}{N}\sum_{n\leq N}\log \frac{\sigma(n)}{n} = A+O\left(\frac{1}{\log x}\right),\quad
\textrm{with}\quad A=\sum_{p\mathrm{ prime}}\alpha(p)\approx 0.4457,
\end{equation*}
where
\begin{equation}
\alpha(p)=\left(1-{1\over p}\right)\sum_{m\geq 1}{1\over p^m}\log\left(1+{1\over p}+\cdots+
{1\over p^m}\right).\end{equation}
\end{proposition}
\noindent
\begin{proof}
Taking the product over all powers $p^m$ dividing $n$, with $p$ prime and 
$m\geq 1$, we have
$$\frac{\sigma(n)}{n}=\prod_{p^m\vert n}\frac{\sigma(p^m)/p^m}{\sigma(p^{m-1})/p^{m-1}},$$
and hence
\begin{eqnarray*}
\sum_{n\leq x}\log\frac{\sigma(n)}{n}&=&\sum_{n\leq x}\sum_{p^m\vert n}\log\frac{p^{m+1}-1}{p(p^m-1)}=\sum_{p^m\leq x}\log\frac{p^{m+1}-1}{p(p^m-1)}\left[\frac{x}{p^m}\right]\\
&=&x\sum_{p^m\leq x}\frac{1}{p^m}\log\frac{p^{m+1}-1}{p(p^m-1)}+O\left(\frac{x}{\log x}\right).
\end{eqnarray*}
For a fixed prime $p$ we have
$$\frac{p^{m+1}-1}{p(p^m-1)}=1+\frac{p-1}{p(p^m-1)}=1+O\left(\frac{1}{p^m}\right),$$
so
\begin{eqnarray*}
& &\hskip-20pt\sum_{p^m\leq x}\frac{1}{p^m}\log\frac{p^{m+1}-1}{p(p^m-1)}=\sum_{p^m\leq x}\frac{1}{p^m}\log\frac{1-1/p^{m+1}}{1-1/p^m}\\
& &\hskip-20pt=-\frac{1}{p}\log\left(1-\frac{1}{p}\right)+
\left(\frac{1}{p}-\frac{1}{p^2}\right)\log\left(1-\frac{1}{p^2}\right)
+\cdots+O\left(\frac{1}{x^2}\right)\\
& &\hskip-20pt=\left(1-\frac{1}{p}\right)\left(\frac{1}{p}\log\left(1+\frac{1}{p}\right)+
\frac{1}{p^2}\log\left(1+\frac{1}{p}+\frac{1}{p^2}\right)+\cdots\right)+O\left(\frac{1}{x^2}\right)
\end{eqnarray*}
and the result follows.
\end{proof}

\medskip\noindent
Theorems \ref{thm:all} and \ref{thm:even} are our main asymptotic results on
the growth of aliquot sequences. Roughly speaking, they state that the growth
factor diverges to 0 when considered over all starting values, whereas confined
to even values it converges, to $\lambda$.
\begin{theorem}\label{thm:all}
\begin{equation*}
{1\over N}\sum_{n=1}^N \log\frac{s(n)}{n}=-e^{-\gamma}\log\log N+O(\log\log\log N).
\end{equation*}
\end{theorem}

\medskip\noindent
\begin{proof}
As $s(n)\geq 1$ for $n>1$ we have $$\frac{s(n)}{n}\geq\frac{1}{p_1}$$
for the smallest prime factor $p_1$ of $n$. Thus
\begin{multline*}
\sum_{1<n\leq x}\log\frac{s(n)}{n}\geq -\sum_{1\leq n\leq x}\log p_1(n)
\geq -\sum_{p\leq x}\log p\sum_{n\leq x\atop q\vert \frac{n}{p}\Rightarrow q\geq p}1\\
\geq - \sum_{p\leq x}\log p\cdot \#\{ m\leq \frac{x}{p} : \mbox{$m$ has its prime factors $\geq p$} \}.
\end{multline*}
By the small sieve: if $p=x^{\frac{1}{u}}$ then
\begin{eqnarray*}
\#\{ m\leq \frac{x}{p} : q\vert m\Rightarrow q\geq p\}&=&
\prod_{q<p}\left(1-\frac{1}{q}\right)\frac{x}{p}\left(1+O(e^{-2u})+\frac{1}{\log x}\right)\\
&=&\frac{e^{-\gamma}x}{p\log p}\left(1+O\left(\frac{1}{\log p}+e^{-2u}\right)\right).
\end{eqnarray*}
Hence
\begin{eqnarray}\label{oneway}
\sum_{1<n\leq x}\log\frac{s(n)}{n}&\geq&-e^{-\gamma}x\sum_{p\leq x}\frac{1}{p}+O\left(x\sum_{p\leq x}\frac{1}{p\log p}+x\sum_{p\leq x}\frac{1}{p}e^{-u}\right)\nonumber\\
&\geq&-e^{-\gamma}x\log\log x+O(x).
\end{eqnarray}
On the other hand, let $M$ be the set of integers of the form $mp\leq x$ where all
prime factors of $m$ are $\geq p\log x$. For such integers we have
$\omega(m)\leq \log(x)/\log(p\log x)$,
$$
\frac{{\sigma(mp)/mp}}{{\sigma(p)/p}}\leq \prod_{q\vert m}\left(1-\frac{1}{q}\right)^{-1}\hskip-7pt\leq \left(1-\frac{1}{p\log x}\right)^{-\omega(m)}\hskip-7pt
\leq 1+O\left(\frac{1}{p\log\log x}\right),
$$
so
\begin{eqnarray*}
\frac{\sigma(mp)}{mp}-1\leq\frac{1}{p}\left(1+O\left(\frac{1}{\log\log x}\right)\right),
\end{eqnarray*}
Then
\begin{eqnarray}\label{otherway}
\sum_{1<n\leq x}\log\frac{s(n)}{n}&=&
\sum_{1<n\leq x}\log\left(\frac{\sigma(n)}{n}-1\right)\nonumber\\
&\leq&\sum_{1<n\leq x}\log\frac{\sigma(n)}{n}-\sum_{n\in M}\left(\log p+O\left(\frac{1}{\log\log p}\right)\right).
\end{eqnarray}
Let $M'$ be the set of integers $mp\in M$ with $p\leq x^{\frac{1}{U}}$
where $U$ is a large fixed number. With $y=e^{(\log\log x)^2}$ we find:
\begin{eqnarray}\label{otherwaysub}
\sum_{n\in M}\log p&\geq & \hskip-3pt\sum_{n\in M'}\log p\geq
\hskip-3pt\sum_{y\leq p\leq x^{\frac{1}{U}}}\hskip-8pt \log p\cdot\#\left\{ m\leq \frac{x}{p}: q\vert m\Rightarrow q>p\log x\right\}\nonumber\\
&\geq&\hskip-3pt\sum_{y\leq p\leq x^{\frac{1}{U}}} \log p\cdot\prod_{q\leq p\log x}\left(1-\frac{1}{q}\right)\frac{x}{p}\left(1+O\left(e^{-U}+\frac{1}{\log x}\right)\right)\nonumber\\
&\geq &\sum_{y\leq p\leq x^{\frac{1}{U}}}\frac{e^{-\gamma}x}{p}\left(1+O\left(e^{-U}+\frac{1}{\log x}\right)\right)\nonumber\\
&\geq &e^{-\gamma}x \left(\log\log x+O(\log\log\log x)\right).
\end{eqnarray}
Combining (\ref{otherwaysub}) and (\ref{otherway}) with (\ref{oneway}) and Proposition (\ref{one}) we obtain
\begin{equation*}\label{two}
\sum_{1\leq n\leq x}\log\frac{s(n)}{n}=-e^{-\gamma}x\log\log x+O(x\log\log\log x).
\end{equation*}
\end{proof}

\begin{theorem}\label{thm:even}
\begin{equation*}
{1\over N}\sum_{n=1}^{N} \log\frac{s(2n)}{2n}=\lambda+ O(1/\log N)
\end{equation*}
where
\begin{equation}
\lambda=\alpha(2)+\sum_{p {\rm \ prime}}\alpha(p) - \sum_{j\geq 1}\frac{1}{j}\left((2\beta_j(2)-1)
\prod_{p\geq 3\atop {\rm prime}}\beta_j(p)\right),\end{equation}
with
\begin{equation}
\alpha(p)=\left(1-{1\over p}\right)\sum_{m\geq 1}{1\over p^m}\log\left(1+{1\over p}+\cdots+
{1\over p^m}\right),\end{equation}
as before, and
\begin{equation}
\beta_j(p)=\left(1-{1\over p}\right)\sum_{m\geq 0}{1\over p^m}\left(1+{1\over p}+\cdots+
{1\over p^m}\right)^{-j}\end{equation}
\end{theorem}

\medskip\noindent
\begin{proof}
Suppose that $J$ is sufficiently large. Then
\begin{eqnarray*}
\sum_{n\leq x}\log\left(\frac{\sigma(2n)}{2n}-1\right)=\sum_{n\leq x}\log\frac{\sigma(2n)}{2n}+\sum_{n\leq x}\log\left(1-\frac{2n}{\sigma(2n)}\right)\\
=\sum_{n\leq x}\log\frac{\sigma(2n)}{2n}-\sum_{j=1}^J\frac{1}{j}\sum_{n\leq x}\left(\frac{2n}{\sigma(2n)}\right)^j+O\left(\frac{x}{(3/2)^J}\right),
\end{eqnarray*}
using (\ref{prop:oneandahalf}).
Proceeding as before we get
\begin{equation}\label{four}
\sum_{n\leq x}\log\frac{\sigma(2n)}{2n}=A^*x+O\left(\frac{x}{\log x}\right),
\end{equation}
where
$$
A^*=\sum_{p}\alpha(p)+\sum_{m\geq 1}\frac{1}{2^{m+1}}\left(\log 2+\log\left(1-\frac{1}{2^{m+1}}\right)\right).
$$
We now use the following result.
\begin{lemma}
Let $f$ be a multiplicative function with
$0\leq f(p^m)\leq 1$ with $1-f(p^m)\ll\frac{1}{p}$ for every prime $p$.
Then
\begin{equation}\label{five}
\sum_{n\leq x} f(n)=x\prod_{p}\left[\left(1-\frac{1}{p}\right)\sum_{m\geq 0}\frac{f(p^m)}{p^m}\right]+O((\log x)^C).
\end{equation}
\end{lemma}
\begin{proof}
Let $g(p^m)=f(p^m)-1\ll\frac{1}{p}$, and $g(n)=\prod_{p^m\vert\vert n}g(p^m)$.
Then
\begin{eqnarray*}
\sum_{n\leq x}f(n)&=&\sum_{n\leq x}\prod_{p^m\vert\vert n}(1+g(p^m))=\\
&=&\sum_{n\leq x}\sum_{d\vert n\atop\gcd(d, n/d)=1}g(d)=\sum_{d\leq x}g(d)\sum_{m\leq x/d\atop \gcd(m,d)=1}1=\\
&=&\sum_{d\leq x}g(d)\frac{\phi(d)}{d}\frac{x}{d}+O(\sum_{d\leq x}g(d)2^{\nu(d)})=\\
&=&x\prod_p\left(1+\sum_{m\geq 1}\frac{g(p^m)}{p^m}\frac{\phi(p^m)}{p^m}\right)+O((\log x)^C).
\end{eqnarray*}
But
\begin{eqnarray*}
1+\sum_{m\geq 1}\frac{g(p^m)}{p^m}\frac{\phi(p^m)}{p^m}&=&1+\left(1-\frac{1}{p}\right)\sum_{m\geq1}\frac{f(p^m)-1}{p^m}=\\
&=&\left(1-\frac{1}{p}\right)\sum_{m\geq0}\frac{f(p^m)}{p^m},
\end{eqnarray*}
and the Lemma follows.
\end{proof}

\medskip\noindent
We apply this to $f_j(n)=\left(\frac{3n}{\sigma(2n)}\right)^j$; then
$$f_j(p^m)=\left(\frac{p^m}{\sigma(p^m)}\right)^j,$$
for $p$ an odd prime, and
$$f_j(2^m)=\left(\frac{3\cdot 2^m}{\sigma(2^{m+1})}\right)^j=\left(\frac{3}{2}\right)^j\left(\frac{2^{m+1}}{\sigma(2^{m+1})}\right)^j.$$
But then by the Lemma
\begin{eqnarray*}
& &\frac{1}{j}\sum_{n\leq x}\left(\frac{2n}{\sigma(2n)}\right)^j=
\frac{1}{j}\frac{1}{(3/2)^j}\sum_{n\leq x}f_j(n)=\\
&=&\frac{x}{j}\frac{1}{(3/2)^j}\prod_p\left\{\left(1-\frac{1}{p}\right)\left(\sum_{m\geq 0}\frac{f_j(p^m)}{p^m}\right)\right\}+O((\log x)^c)=\\
&=&\frac{x}{j}(2\beta_j(2)-1)\left(\prod_{p\geq3\atop\textrm{prime}}\beta_j(p)\right)+O((\log x)^c).
\end{eqnarray*}

\medskip\noindent
Collecting the information we get
\begin{equation*}\label{seven}
\sum_{n\leq x}\log\left(\frac{\sigma(2n)}{2n}-1\right)= \lambda x+O\left(\frac{1}{\log x}\right),
\end{equation*}
and the result follows.
\end{proof}
\begin{example}\rm
Although most of the rest of this paper is devoted to a numerical
estimate for $\lambda$, necessary because of the behaviour of $\beta$,
it is easy to {\it see} the convergence of $\lambda$. If we
sum $s(n)/n$ for the first even values of $n$, up to $N$, we get the following.
\begin{table}[h]
\begin{tabular}{|c|c|c|}\hline
$N$ & $\sum_{2n\leq N}s(2n)/(2n)$ \\\hline
$10^2$ & $-0.0567457527...$ \\
$10^3$ & $-0.0356519058...$ \\
$10^4$ & $-0.0335201796...$ \\
$10^5$ & $-0.0332873082...$ \\
$10^6$ & $-0.0332626444...$ \\
$10^7$ & $-0.0332598642...$ \\
$10^8$ & $-0.0332595156...$ \\
$3.953\cdot10^9$ & $-0.0332597045...$ \\\hline
\end{tabular}
\end{table}
\end{example}
\section{Computing $\alpha$}
\noindent
By definition,
$$\lambda=2\alpha(2)+\sum_{p\geq 3\atop \mbox{\tiny\ \ prime}}\alpha(p)-\sum_{j\geq 1}\frac{1}{j}\left((2\beta(2)-1)\cdot\prod_{p\geq 3 \atop\mbox{\tiny\ \ prime}}\beta_j(p)\right).$$
In this section we will compute a good approximation and upper bound for:
$$\alpha=2\alpha(2)+\sum_{p\geq 3\atop p\ \rm{prime}}\alpha(p).$$
First note that for any prime $p$
$$
\left(1-\frac{1}{p}\right)\sum_{m\geq 1}\frac{1}{p^m}\log\left(1+\frac{1}{p}+\cdots\frac{1}{p^m}\right)=
\sum_{m\geq 1}\frac{1}{p^m}\log\left(\frac{1+p+\cdots p^m}{p+\cdots+p^m}\right).
$$
Since for all $m\geq 1$
$$
\frac{1+p+\cdots+p^m}{p+\cdots+p^m}=1+\frac{1}{p+\cdots+p^m}\leq 1 + \frac{1}{p^m}
$$
and $\log(1+x)\leq x$ for all $x>0$, we find that the tail
$$
\sum_{m=M+1}^\infty \frac{1}{p^m}\log\left(\frac{1+p+\cdots+p^m}{p+\cdots+p^m}\right)
$$
is bounded by
$$
\frac{1}{p^{M+1}}\left(1+\frac{1}{p}+\frac{1}{p^2}+\cdots\right)\cdot\log\left(1+\frac{1}{p^{M+1}}\right)\leq \frac{p}{p-1} \left(\frac{1}{p^{M+1}}\right)^2.
$$
We will denote this bound by
$$A(p, M)=\frac{p}{p-1} \left(\frac{1}{p^{M+1}}\right)^2.$$
It also follows that, for any $N\geq 1$
$$\sum_{p> N\atop p\ {\rm prime}}\alpha(p)\leq \sum_{p> N\atop p\ {\rm prime}}
\sum_{m\geq 1}\frac{1}{p^m}\log\left(1+\frac{1}{p^m}\right)\leq
\sum_{n>N}\frac{1}{n}\log\left(1+\frac{1}{n}\right).$$
But
$$\sum_{n>N}\frac{1}{n}\log\left(1+\frac{1}{n}\right)\leq \int_{N}^{\infty} \frac{1}{x}\log\left(1+ \frac{1}{x}\right)\ dx=\int_{0}^{\frac{1}{N}} \frac{1}{z}\log\left(1+ z\right)\ dz$$
is clearly bounded by $\frac{1}{N}$.
\begin{theorem}
For any $N>2$ and $L,M>1$:
\begin{eqnarray*}
\alpha&\leq& \sum_{m=1}^L \frac{1}{2^m}\log\left(1+\frac{1}{2}+\cdots+\frac{1}{2^m}\right)+2A(2,L)+\\
&+&\sum_{p\leq N\atop p\ {\rm odd prime}}\sum_{m=1}^M \frac{1}{p^m}\log\left(\frac{1+\cdots+p^m}{p+\cdots+p^m}\right)+\sum_{3\leq p\leq N\atop p\ {\rm prime}}A(p,M)+\frac{1}{N}.
\end{eqnarray*}
\end{theorem}
As the sums in this theorem are all finite this gives us an effective way to compute
an upper bound on $\alpha$.
\begin{example}\rm
In the table below we have listed the outcome of some computations
for $\alpha$ with $L=M=15$. These computations were done within
half an hour (including the primality tests for all primes up to $10^8$)
on an ordinary PC, using \Magma.
\begin{table}[h]
\begin{tabular}{|c|c|c|}\hline
$N$ & sums &  error bound\\\hline
$10^4$ & $0.6983072233...$ & $1.0000093132...\cdot 10^{-4}$\\
$10^5$ & $0.6983162365...$ & $1.0000931323...\cdot 10^{-5}$\\
$10^6$ & $0.6983169710...$ & $1.0009313233...\cdot 10^{-6}$\\
$10^7$ & $0.6983170329...$ & $1.0093132338...\cdot 10^{-7}$\\
$10^8$ & $0.6983170383...$ & $1.0931323384...\cdot 10^{-8}$\\\hline
\end{tabular}
\end{table}
\end{example}
\begin{corollary}
$\alpha<0.69831705$.
\end{corollary}
\section{Computing $\beta$}
\noindent
Our next goal is to compute a lower bound for
\begin{equation}\label{def:beta}
\beta=\sum_{j\geq 1}\frac{1}{j}\left((2\beta_j(2)-1)\prod_{p> 2 \atop\mbox{\tiny\ \ prime}}\beta_j(p)\right),
\end{equation}
where for every prime $p$
$$\beta_j(p)=\left(1-\frac{1}{p}\right)\left(1+\sum_{m\geq1}\frac{1}{p^m}\left(\frac{1}{1+\frac{1}{p}+\frac{1}{p^2}+\cdots+\frac{1}{p^m}}\right)^j\right).$$
Observe that for $j\geq 1$
\begin{eqnarray*}
\beta_j(p)&=&\sum_{m\geq 0}\frac{1}{p^m}\left(\frac{p^m}{\sigma(p^m)}\right)^j-
\sum_{m\geq 0}\frac{1}{p^{m+1}}\left(\frac{p^m}{\sigma(p^m)}\right)^j\\
&=&1+\sum_{m\geq1}\frac{1}{p^m}\frac{(p^m+\cdots+p^{2m-1})^j-
(p^{m-1}+\cdots+p^{2m-1})^j}{\sigma(p^m)^j\sigma(p^{m-1})^j}\\
&=&1-\sum_{m\geq1}\left(\frac{1}{p^m}\frac{\sum_{k=1}^j {j\choose k}(p^{m-1})^k(p^m\sigma(p^{m-1}))^{j-k}}{\sigma(p^m)^j\sigma(p^{m-1})^j}\right)\\
\end{eqnarray*}
\begin{eqnarray*}
&=&1-\sum_{m\geq1}\left(\frac{1}{p^m}\sum_{k=1}^j\frac{ {j\choose k}p^{mj-k}\sigma(p^{m-1})^{j-k}}{\sigma(p^m)^j\sigma(p^{m-1})^j}\right)\\
&=&1-\sum_{m\geq1}\left(\frac{1}{p^m}\left(\frac{p^m}{\sigma(p^m)}\right)^j\sum_{k=1}^j\frac{{j\choose k}}{(p\sigma(p^{m-1}))^k}\right)\\
&=&\sum_{m\geq0}(-1)^{\nu(p^m)}g_j(p^m)h_j(p^m),
\end{eqnarray*}
where the multiplicative function $\nu(n)$ denotes the number of 
different prime divisors of $n$, and we set for non-negative $m$
$$g_j(p^m)=\frac{1}{p^m}\left(\frac{p^m}{\sigma(p^m)}\right)^j,$$
(in particular $g_j(1)=1$), and for positive $m$
$$h_j(p^m)=\sum_{k=1}^j\frac{{j\choose k}}{(p\sigma(p^{m-1}))^k}\ ,$$
while $h_j(1)=1$ by definition.

If we extend our definition to
$$g_j(n)=\frac{1}{n}\left(\frac{n}{\sigma(n)}\right)^j$$
for any positive integer $n$, we find
$$g_j(n)=\prod_{p^m\parallel n}g_j(p^m)$$
where $n=\prod_{p^m\parallel n}p^m$ is the factorization of $n$.
We also define $h_j$ for all positive integers $n$ by multiplicativity
$$h_j(n)=\prod_{p^m\parallel n}h_j(p^m).$$
If we define, for composite $n$,
\begin{equation}\label{def:betacomp}
\beta_j(n)=(-1)^{\nu(n)}g_j(n)h_j(n);
\end{equation}
then for every positive integer $n$, we find
$$\prod_{p\geq 2\atop p\ \rm prime}\beta_j(p)=\sum_{n\in\N}\beta_j(n).$$
This way, the infinite product in the definition of $\beta$, (\ref{def:beta}),
is replaced by an infinite sum; note, however, the slight complication caused
by the factor $2\beta_j(2)-1$ rather than $\beta_j(2)$ in this definition.
Since
\begin{equation}\label{beta2}
2\beta(2)-1=\sum_{m\geq1}\frac{1}{2^m}
\left(\frac{1}{1+\frac{1}{2}+\frac{1}{2^2}+\cdots+\frac{1}{2^m}}\right)^j=\sum_{m\geq1}g_j(2^m)
\end{equation}
we obtain
\begin{equation}
\beta=\sum_{j\geq1}\frac{1}{j}\left(\sum_{n\in\N\atop n\ \rm{even}\ }
\beta_j^*(n)\right),
\end{equation}
where
\begin{equation}\label{def:betastar}
\beta_j^*(n)=g_j(2^k)\beta_j(n_o)\quad\textrm{for}\quad n=2^kn_o,\textrm{ with $n_o$ odd.}
\end{equation}

To get a lower bound for $\beta$ we will replace these sums by finite summations
for bounded $j$ and $n$, and bound the remaining terms. As clearly $\beta_j(p)>0$
for odd $p$, and so is $2\beta_j(2)-1$ by \eqref{beta2}, we see from the definition
that for any $J\geq 1$
$$\beta\geq\sum_{j=1}^J
\frac{1}{j}\left(\sum_{n\in\N\atop n\ \rm{even}\ }\beta_j^*(n)\right).$$
Define for real $e>0$
$$S_{j,e}=\left\{ n\in\N:\ h_j(n)>\frac{1}{n^e} \right\}.$$
Since
\begin{eqnarray*}
\sum_{n=2^kn_o\textrm{\ even}\atop n_o> N,\ n_o\notin S_{j,e}}\beta_j^*(n)&=&
\sum_{n=2^kn_o\textrm{\ even}\atop n_o> N,\ n_o\notin S_{j,e}}\frac{1}{n}\left(\frac{n}{\sigma(n)}\right)^jh_j(n_o)\leq\\
&\leq&\sum_{n_o> N\atop\textrm{odd}}\left(\sum_{k=1}^\infty\frac{1}{2^kn_o}\left(\frac{2^kn_o}{\sigma(2^kn_o)}\right)^j\right)h_j(n_o)\leq\\
&\leq& \left(\frac{2}{3}\right)^j\frac{1}{2}\int_{x=N}^\infty\frac{1}{x^{1+e}}dx=\frac{1}{2e}\frac{1}{N^e}\left(\frac{2}{3}\right)^j,
\end{eqnarray*}
we immediately obtain the following result.

\begin{lemma}\label{lem:lbnd}
For any $J\geq 1$ and even $N_j>1$ (for $j=1,2\ldots,J$):
$$\beta \geq\left(\sum_{j=1}^J\frac{1}{j} \sum_{n\leq N_j\ \mathrm{ or }\ n\in S_{j,e}\atop n\mathrm{\ even}}\beta_j^*(n)\right)
-\sum_{j=1}^J\frac{1}{2je N_j^e}\left(\frac{2}{3}\right)^j.$$
\end{lemma}
\noindent
Define for $e,c>0$
$$
T_j^{e,c} = \left\{ (p,m) : h_j(p^m) \geq  \frac{1}{ c\cdot (p^m)^{e} },\ \textrm{$p$ prime, $m\in\Z$ with $m\geq 1$} \right\}
$$
Now $T_j^{e,c}$ is finite whenever $0<e<1$ (let $p^m\to\infty$ in the definition). Put
$$
M_{j,e}=\max_{n\in\Z_{\geq1}} h_j(n)n^{e};
$$
note that $M_{j,e}\geq 1$.
\begin{lemma} Let $0< e< 1$. If $n\in S_{j,e}$ then for $m\geq 1$
$$p^m\parallel n\quad\Rightarrow\quad (p, m)\in T_j^{e, M_{j,e}};$$
in particular, $S_{j,e}$ is finite.
\end{lemma}
\noindent
\begin{proof}
Let $M=M_{j,e}$.
Suppose that $n\in S_{j,e}$, and write $n=n_1\cdot n_2$ with
$\gcd(n_1,n_2)=1$, such that
if $p^m\parallel n_1$ then $(p, m)$ is in $T_j^{e,M}$ and
if $p^m\parallel n_2$ then $(p, m)\notin T_j^{e,M}$.

If $n_2>1$ then
$$
h_j(n)n^{e}= h_j(n_1)n_1^{e} \cdot h_j(n_2)n_2^{e} \leq M \cdot \frac{1}{M} = 1,
$$
contradicting our assumption $n\in S_{j,e}$. Thus $n_2=1$ and the lemma is proved.
\end{proof}

\medskip\noindent
This means that the inner sums in Lemma (\ref{lem:lbnd}) are finite.
\begin{algorithm}
{\rm This results in the following method for computing a lower bound for $\beta$.}
\begin{itemize}
\item[(1)] Choose $J\geq 1$, and 
perform the following three steps for $j=1, 2, \ldots, J$.
\item[(2)] Choose an even integer $N_j>1$ large enough.
\item[(3)] Compute $\frac{1}{j}\sum_{n\leq N_j\atop\textrm{even}}\beta_j^*(n)$, for example
from the definitions (\ref{def:betastar}) and (\ref{def:betacomp}).
\item[(4)] Determine the set $S_{j,e}$ as follows:
\begin{itemize}
\item[(4a)] Choose $e=e_j$ and determine $T_{j,e,1}$. Then compute $M_{j,e}=\max_n h_j(n) n^e$ by choosing the product of the worst value for $(p,m)\in T_{j,e,1}$ for each prime $p$ occurring in this set. 
\item[(4b)] Choose $c=c_j$ and determine $T_{j,e,c}$.
\item[(4c)] Determine $S_{j,e}$, that is, the positive integers $n$ built up
from prime powers $p^m$ with $(p,  m)$ in $T_{j,e,c}$, for which $h_j(n)>1/n^e$.
\end{itemize}
and compute $\frac{1}{j}\sum_{n\in S_{j,e}\atop{n>N \textrm{even}}}\beta_j^*(n)$.
\item[(5)] Compute the error term $(2jeN_j^e)^{-1}\cdot\left(\frac{2}{3}\right)^j$.
\item[(6)] Take the sum of results from Steps 3 and 4, and subtract the sum of
the results of Step 5, taken over $j=1, 2, \ldots, J$.
\end{itemize}
\end{algorithm}
\begin{example}\rm
\begin{table}[hb]
\begin{tabular}{|c|l|r|l|r|r|}\hline
$j$ &  \hfill$e$\hfill    &\hfill $\#S$\hfill      &\hfill main\hfill &\hfill contribution $S$\hfill &\hfill error\hfill \\\hline
$1$ & $1$    & $0$       & $0.508058$  & & $4.18\cdot 10^{-12}$ \\
$2$ & $0.75$  & $71678431$     & $0.134230$  &$4.7096\cdot 10^{-12}$ & $2.99\cdot 10^{-9}$\\
$3$ & $0.60$  & $139189128$    & $0.048944$ & $8.949\cdot 10^{-12}$ & $3.276\cdot 10^{-7}$\\
$4$ & $0.48$  & $93183633$   & $0.020684$ & $9.7488\cdot10^{-12}$ & $2.462\cdot 10^{-6}$\\
$5$ & $0.35$ & $10201152$ & $0.009564$& $-9.1679\cdot 10^{-12}$ & $2.66\cdot 10^{-5}$\\
$6$ & $0.28$ & $27662520$  & $0.004706$& $-1.95\cdot 10^{-12}$ & $7.89\cdot 10^{-5}$\\
$7$ & $0.20$ & $24415897$ & $0.002425$& $-3.315\cdot 10^{-12}$ & $3.31\cdot 10^{-4}$\\
$8$ & $0.15$ & $65291514$ & $0.001295$ & $5.907\cdot 10^{-12}$ & $7.26\cdot 10^{-4}$ \\
$9$ & $0.03$ & $7466778$ & $0.000711$ & $-9.511\cdot 10^{-12}$& $2.59\cdot 10^{-2}$ \\\hline
$\sum_{j=1}^8$ &     & & $0.729906$  &$<10^{-10}$ & $1.1625\cdot 10^{-4}$\\\hline
\end{tabular}
\end{table}

Take $N=10^9$.  The table lists for $j=1, 2, \ldots, 9$ the values of
$e$, $\#S$, $M$, the main term from Step (3) in the algorithm,
the contribution from $S$ in Step (4) and the error term from Step (5).

As an indication of the size of the numbers involved: the largest element
of $S_{2, 0.75}$ is the 24-digit product of the 18 primes less than 62.
Since the error term for $j=9$ exceeds the contribution of the main term, we
have not included this one in the final sum.

As a consequence, we find that $\beta\geq 0.728743$.
\end{example}
\begin{corollary}
$$\lambda<-0.030.$$
\end{corollary}


\begin{thebibliography}{10}
\bibitem{Magma} Wieb Bosma, John Cannon, and Catherine Playoust,
{\it The Magma algebra system. I. The user language}, J. Symbolic Comput.,
{\bf 24} (1997), 235--265.
\bibitem{Briggs} Keith Briggs, {\it Abundant numbers and the Riemann hypothesis},
Experiment.\ Math.\ {\bf 15} (2006), 251--256.
\bibitem{Davenport} Harold Davenport, {\it \"Uber numeri abundantes},
Preuss.\ Akad.\ Wiss.\ Sitzungsberichte (1933), 830--837.
\bibitem{Deleglise} M. Del\'eglise, {\it Bounds for the density of abundant
integers}, Experiment.\ Math. {\bf 7} (1998), 137--143.
\bibitem{Dirichlet} G. Lejeune Dirichlet, {\it \"Uber die Bestimmung
der mittleren Werthe in der Zahlentheorie}, Abhandlungen der K\"oniglich
Preussischen Akademie von 1849, 69--83 (= {\it Werke}, zweiter Band, 49--66).
\bibitem{Granville} Andrew Granville, {\it personal communication}.
\bibitem{GuySelfridge} Richard K. Guy, J.\,L. Selfridge,
{\it What drives an aliquot sequence?}, Math.\ Comp. {\bf 29} (1975), 101--107;
Corrigendum:  Math.\ Comp. {\bf 34} (1980), 319--321.
\bibitem{Kobayashi} Mitsuo Kobayashi, Paul Pollack, Carl Pomerance,
{\it On the distribution of sociable numbers}, J. Number Theory {\bf 129} (2009), 1990--2009.
\bibitem{Lagarias} J.$\,$C. Lagarias, {\it An elementary problem equivalent
to the Riemann hypothesis}, Amer.\ Math.\ Monthly {\bf 109} (2002), 534--543.
\bibitem{Moews} {\tt http://djm.cc/sociable.txt}
\bibitem{Petermann} Y.-F.$\,$S. P\'etermann, {\it An $\Omega$-theorem for an
error term related to the sum-of-divisors functions}, Monatsh.\ Math.\ {\bf 103} (1987), 145--157. 
\bibitem{Robin} G. Robin, {\it Grandes valeurs de la fonction somme des
diviseurs et hypoth\`ese de Riemann}, J.\ Math.\ Pures Appl.\ {\bf 63} (1984),
187--213.
\bibitem{Weingartner} Andreas Weingartner, {\it The distribution functions
of $\sigma(n)/n$ and $n/\phi(n)$}, Proc.\ Amer.\ Math.\ Soc. {\bf 135} (2007),
2677--2681.
\bibitem{Wojtowicz} Marek W\'jtowicz, {\it Robin's inequality and the Riemann
hypothesis}, Proc.\ Japan Acad.\ {\bf 83} Ser.\ A (2007), 47--49.
\bibitem{Wunderlich}M.\,C. Wunderlich, {\it The aliquot project: an application of job chaining in number theoretic computing},
Proceedings of the third ACM symposium on Symbolic and algebraic computation
(1976), 276--284.
\end{thebibliography}
\end{document}